\documentclass{amsart}
\usepackage{all2021}
\newcommand{\Griso}{\mathrm{Gr}_{\mathrm{iso}}}
\newcommand{\HG}{\mathcal{H}^\mathrm{Gor}}
\newcommand{\cH}{\mathcal{H}}

\newcommand{\fm}{\mathfrak{m}}

\usepackage{color}

\begin{document}

\title{On the quadratic equations for odeco tensors}
\author{Benjamin Biaggi}
\address{Mathematical Institute, University of Bern, Sidlerstrasse 5,
3012 Bern, Switzerland}
\email{benjamin.biaggi@unibe.ch}

\author{Jan Draisma}
\address{Mathematical Instittue, University of Bern, Sidlerstrasse 5,
3012 Bern, Switzerland; and Department of Mathematics and Computer
Science, Eindhoven University of Technology, P.O.~Box 513, 5600 MB
Eindhoven, The Netherlands}
\email{jan.draisma@unibe.ch}

\author{Tim Seynnaeve}
\address{Mathematical Institute, University of Bern, Sidlerstrasse 5,
3012 Bern, Switzerland}
\email{tim.seynnaeve@unibe.ch}

\thanks{The authors were partly funded by Vici grant 639.033.514 from the
Netherlands Foundation for Scientific Research and by project
grant 200021\_191981 from the Swiss National Science Foundation.}

\maketitle

\begin{center}
{\em To Giorgio Ottaviani, on the occasion of his 60th birthday.}
\end{center}

\begin{abstract}
Elina Robeva discovered quadratic equations satisfied by orthogonally
decomposable (``odeco'') tensors. Boralevi-Draisma-Horobe\c{t}-Robeva
then proved that, over the real numbers, these equations {\em
characterise} odeco tensors. This raises the question to what extent
they also characterise the Zariski-closure of the set of odeco tensors
over the complex numbers. In the current paper we restrict
ourselves to {\em symmetric tensors of order three}, i.e., of format 
$n \times n \times n$. By providing an explicit
counterexample to one of Robeva's conjectures, we show that for $n
\geq 12$, 
these equations do not suffice. Furthermore, in the open subset where the linear span of the
slices of the tensor contains an invertible matrix, we show that Robeva's
equations cut out the limits of odeco tensors for dimension $n \leq 13$,
and not for $n \geq 14$ on.  To this end, we show that Robeva's equations
essentially capture the Gorenstein locus in the Hilbert scheme of $n$
points and we use work by Casnati-Jelisiejew-Notari on the
(ir)reducibility of this locus. 
\end{abstract}

\section{Introduction}

In \cite{Robeva16Orthogonal}, Robeva discovered quadratic
equations satisfied by orthogonally decomposable (odeco) tensors. In
\cite{Boralevi17Orthogonal}, it is proved that over the real numbers,
these quadratic equations in fact {\em characterise} odeco tensors.

This raises the question whether Robeva's equations also define
(the Zariski closure of) the set of complex odeco tensors. Indeed,
Robeva conjectured that they might even generate the prime ideal
of this Zariski closure, at least in the case of symmetric tensors
\cite[Conjecture 3.2]{Robeva16Orthogonal}. She proved this stronger
statement when the ambient space has dimension at most $3$ \cite[Figure
2]{Robeva16Orthogonal}. In general, however, the answer to the (weaker)
question is no, as already pointed out by Koiran in \cite{Koiran21}. In
this short paper, based on the first author's Master's thesis, we give
an explicit symmetric tensor in $(\CC^{12})^{\otimes 3}$ that satisfies
Robeva's equations but is not approximable by complex odeco tensors. We
do not know whether $12$ is the minimal dimension for which this happens,
but we show that if we impose a natural, additional {\em open} condition
on the tensor, then Robeva's equations characterise the Zariski closure
of the odeco tensors precisely up to dimension $13$.

A key idea in \cite{Boralevi17Orthogonal} is to associate an algebra $A$
to a symmetric three-tensor $T$ and to realise that Robeva's equations
express the associativity of that algebra. Furthermore, the symmetry of
the tensor implies that $A$ is commutative and that the bilinear form is
an invariant form on $A$; see below for definitions.  If, in addition,
$A$ contains a unit element---this turns out to be an open condition on
$T$---then $A$ is a {\em Gorenstein algebra}. Consequently, we can use
the results of \cite{Casnati15} on (ir)reducibility of the Gorenstein
locus in the Hilbert scheme of points in affine space to study the
variety defined by Robeva's equations.

In the opposite direction, we use this relation between algebras and
tensors to give an elementary proof that the Gorenstein locus in the
Hilbert scheme of $n$ points in $\AA^n$ has a dimension that grows
as $\Theta(n^3)$. This seems surprising at first, since the component
containing the schemes consisting of $n$ distinct reduced points has
dimension only $n^2$; on the other hand, it is well-known that the
dimension of the Hilbert scheme itself does grow as a cubic function
of $n$. 

This relation between algebras and tensors is, of course, not new: a
bilinear multiplication on $V$ can be thought of as an element of $V^*
\otimes V^* \otimes V$, and in the presence of a bilinear form on $V$,
the copies of $V^*$ may be identified with $V$. Further properties
of the algebra, such as associativity, cut out subvarieties of the
corresponding tensor space.  A classical reference for varieties of
algebras is \cite{Flanigan68}, where the term {\em algebraic geography}
is coined.  Another, more closely related paper is \cite{PoonenModuli},
whose Remark 4.5 is closely related to Lemma~\ref{lm:HG0}, and together
with \cite[Theorem 9.2]{PoonenModuli} gives the cubic lower bound on the
dimension of the Hilbert scheme mentioned above. Finally, we note that
the Zariski closure of the odeco tensors consists of tensors of
minimal border rank; equations for these are studied in the recent paper
\cite{Jelisiejew22}. In particular, \cite[Proposition 1.4]{Jelisiejew22},
which states that, in the $1$-generic locus, the $A$-Strassen equations
are sufficient to characterise tensors of minimal border rank is closely
related to our Theorem~\ref{thm:XY2}.

\subsection{Organisation of this paper}
In Section~\ref{sec:Setup}, we introduce the fundamental notions
of this paper, including Robeva's equations, of which we show that
they are the {\em only} quadrics that vanish on odeco tensors. We 
also state
our main results (Theorems~\ref{thm:XY1} and~\ref{thm:XY2}). In
Section~\ref{sec:Decomposition} we extend the well-known decomposition of
finite-dimensional unital algebras into products of local algebras to the
non-unital case. 
In Section~\ref{sec:Component} we recall that the Zariski
closure of the odeco tensors is a component of the variety cut
out by Robeva's equations; this was already established in \cite[Lemma
3.7]{Robeva16Orthogonal}. In Section~\ref{sec:Many} we show that there
are many weakly odeco tensors; combined with later results, this gives
a lower bound on the dimension of the Gorenstein locus in the Hilbert
scheme of $n$ points in $\AA^n$. In Section~\ref{sec:Unitalisation}
we show how to unitalise algebras along with an invariant bilinear
form to turn them into a local Gorenstein algebra, and vice versa. In
Section~\ref{sec:Nilpotent} we use this construction to motivate the
search for nilpotent counterexamples to Robeva's conjecture. Then,
in Section~\ref{sec:XY1} we find such a counterexample for $n=12$. In
Section~\ref{sec:XY2} we make the connection with the Gorenstein locus in
the Hilbert scheme of $n$ points and prove our second main result---that a
version of Robeva's conjecture holds in the (open) unital locus precisely
up to $n=13$. Finally, in Section~\ref{sec:Cubic} we show that the
dimension $d(n)$ of that Gorenstein locus is lower-bounded bounded by
a cubic polynomial in $n$. Since it is also trivially upper-bounded by
such a polynomial, we have that $d(n)=\Theta(n^3)$.

\section{Set-up} \label{sec:Setup}

\subsection{Weakly and strongly odeco tensors}
Let $V_\RR$ be a finite-dimensional real vector space equipped with a
positive-definite inner product $(.|.)$.

\begin{de}
A symmetric tensor $T \in S^3 V_\RR \subseteq V_\RR \otimes V_\RR \otimes
V_\RR$ is called {\em orthogonally decomposable} (odeco, for short) if,
for some integer $k \geq 0$, $T$ can be written as
\[ T=\sum_{i=1}^k v_i^{\otimes 3} \]
where $v_1,\ldots,v_k \in V_\RR$ are nonzero, pairwise orthogonal 
vectors. We write $Y(V_\RR) \subseteq S^3 V_\RR$ for the set of odeco
tensors.
\end{de}

Positive-definitiveness of the form implies that $(v_i|v_i)>0$ for each
$i$, so that $v_1,\ldots,v_k$ are linearly independent. Hence $k \leq
n:=\dim(V_\RR)$, and $Y(V_\RR)$ is a semi-algebraic set of dimension
at most $\binom{n}{2} + n$: the dimension of the orthogonal group plus
$n$ degrees of freedom for scaling. It turns out that $Y(V_\RR)$ is in
fact the set of real points of an algebraic variety defined by quadratic
equations that we will discuss below. Furthermore, it is easy to see
that the $k$ and the $v_i$ in the decomposition of an odeco tensor $T$
are unique, which implies that the dimension of $Y(V_\RR)$ is precisely
$\binom{n}{2} + n$.

We now set $V:=\CC \otimes V_\RR$ and extend $(.|.)$ to a
complex symmetric bilinear form ({\em not} a Hermitian form;
that setting is studied in \cite{Boralevi17Orthogonal}
under the name {\em udeco}). 

\begin{de}
A symmetric tensor $T \in S^3 V \subseteq V \otimes V \otimes V$ (where the tensor product is over
$\CC$) is {\em weakly odeco} if $T$ can be written as
\[ T=\sum_{i=1}^k v_i^{\otimes 3} \]
where the $v_i$ are nonzero pairwise orthogonal vectors. It is called {\em
strongly odeco} if $T$ admits such a decomposition where, in addition,
$(v_i|v_i) \neq 0$. We write $Y(V) \subseteq S^3V$ for the Zariski
closure of the set of strongly odeco tensors.
\end{de}

\begin{re}
The set called $\mathrm{SODECO}_n(\CC)$ in \cite{Koiran21} consists of
the strongly odeco tensors in $S^3 \CC^n$. Koiran proves that these are
precisely the set of symmetric tensors whose $n \times n$ slices are diagonalisable
and commute.
\end{re}

As pointed out above, every element of $Y(V_\RR)$, regarded as an element
of $S^3 V$, is strongly odeco. Since the real orthogonal group $O(V_\RR)$
is Zariski dense in the complex orthogonal group $O(V)$, $Y(V)$ is 
the Zariski closure of $Y(V_\RR)$. On the other hand, due to the presence
of isotropic vectors and higher-dimensional spaces in $V$, the set of
weakly odeco tensors strictly contains the set of strongly odeco tensors;
we will return to this theme shortly.  First we give the easiest example
that shows the need for a Zariski closure in the definition of $Y(V)$.

\begin{ex} \label{ex:DimTwo}
Consider the vector space $V=\CC^2$ equipped with the symmetric bilinear
form for which $(e_1|e_2)=1$ and all other products are zero. Then
\[ S=e_1 \otimes e_2 \otimes e_2 + e_2 \otimes e_1 \otimes
e_2 + e_2 \otimes e_2 \otimes e_1 
= \frac{1}{2} \lim_{t \to 0}[ 
(t^2 e_1 + t^{-1} e_2)^{\otimes 3} + 
(t^2 e_1 - t^{-1} e_2)^{\otimes 3} ]
\]
shows that $S$ is a limit of strongly odeco tensors. So $S \in Y(V)$, but
$S$ is not strongly odeco, because its tensor rank is $3$ rather than $2$.

(Of course, since all nondegenerate symmetric bilinear forms on a
finite-dimen\-sional complex vector space are equivalent, we
could have changed coordinates such that the bilinear form
is the standard form.)
\end{ex}

\subsection{Commutative algebras from symmetric tensors}

We now associate an algebra structure on $V$ to a tensor. 

\begin{de}
We identify $V$ with $V^*$ via
the map $v \mapsto (v|.)$. Then any element $T \in
V^{\otimes 3}$ can also be regarded as an element of $V^* \otimes V^*
\otimes V$, hence a bilinear map $\mu_T:V \times V \to V$.
We call $V$ with $\mu_T$ the {\em algebra associated to
$T$}. 
\end{de}

If $T$ is symmetric, then, first, $\mu_T$ is commutative,
and second, $\mu_T$ satisfies 
\[ (\mu_T(x,y) | z) = (x|\mu_T(y,z)), \]
i.e., the bilinear form $(.|.)$ is {\em invariant for the
multiplication $\mu_T$}. When $T$ is fixed in the context,
then we will often just write $xy$ instead of $\mu_T(x,y)$. 

\subsection{Odeco implies associative}

\begin{prop} \label{prop:Associative}
If $T \in S^3 V$ is weakly odeco, then $\mu_T$ is
associative. 
\end{prop}

This was observed in \cite{Boralevi17Orthogonal} in the real case, but the
argument easily generalises, as follows. 

\begin{proof}
Write 
\[ T=\sum_{i=1}^k v_i^{\otimes 3} \]
where the $v_i$ are pairwise orthogonal. Let $x,y,z \in V$. Then 
\begin{align*}
\mu_T(x,\mu_T(y,z))
=& \mu_T(x,\sum_{i=1}^k (y|v_i)(z|v_i)v_i)
= \sum_{j=1}^k \sum_{i=1}^k (y|v_i)(z|v_i)(x|v_j)(v_i|v_j)v_j \\
=& \sum_{i=1}^k (y|v_i)(z|v_i)(x|v_i)(v_i|v_i)v_i  
\end{align*}
where the last equality uses that $(v_i|v_j)=0$ whenever $i
\neq j$. A similar computation for $\mu_T(\mu_T(x,y),z)$
yields the exact same result.
\end{proof}

\subsection{Robeva's equations}
For any fixed $x,y,z$ (e.g. chosen from a basis of $V$), the condition
that $\mu_T(x,\mu_T(y,z))$ equals $\mu_T(\mu_T(x,y),z)$ translates into
$n$ quadratic equations for $T$. 
All these equations
together, with varying $x,y,z$, are called {\em Robeva's equations}.

\begin{de}
We denote by $X(V) \subseteq S^3 V$ the affine variety defined by
Robeva's equations, i.e., 
\[ X(V):=\{T \in S^3 V \mid \mu_T \text{ is associative}\}.
\qedhere
\]
\end{de}

To prove that Robeva's equations are {\em all} quadratic equations
satisfied by strongly odeco tensors, we reinterpret them as follows. 
The condition that $\mu_T(x,\mu_T(y,z))$ equals $\mu_T(\mu_T(x,y),z)$ means that for every $w \in V$, we have $$(\mu_T(x,\mu_T(y,z))|w) = (\mu_T(\mu_T(x,y),z)|w),$$ which can be rewritten as 
\begin{equation} \label{eq:Robevaxyzw}
(\mu_T(y,z)|\mu_T(x,w)) = (\mu_T(x,y)|\mu_T(z,w)).
\end{equation}
In other words: Robeva's equations precisely express that the 4-linear map
\begin{align*}
T \bullet T: V^{4} &\to \CC \\
 (x,y,z,w) &\mapsto (\mu_T(x,y)|\mu_T(z,w))
\end{align*}
is invariant under arbitrary permutations of $(x,y,z,w)$. This was, in
fact, Robeva's original description  of these
quadratic equations in \cite{Robeva16Orthogonal}. 

\begin{prop}
	The only quadratic equations vanishing on $Y(V)$ are Robeva's equations.
\end{prop}
\begin{proof}
If we consider the natural map 
\begin{align*}
	S^2(S^3 V) &\to S^2 (S^2 V) \\
	(v_1 \otimes v_2 \otimes v_3) \otimes (w_1 \otimes w_2 \otimes
	w_3) &\mapsto (v_1|w_1) (v_2 \otimes v_3) \otimes (w_2 \otimes
	w_3), 
\end{align*}
then for any  $T \in S^3V$, we can identify $T \bullet T$ with the image of $T$ under the composition
\begin{equation} \label{eq:RobevaVeronese}
S^3V \xrightarrow{\nu_2} S^2S^3V \to S^2S^2V,
\end{equation}
where $\nu_2$ is the second Veronese embedding. 
Then $T$ satisfies Robeva's equations if and only if $T \bullet T \in S^4 V \subseteq S^2S^2V$.

Since quadratic equations on $Y(V)$ correspond to linear equations on $\nu_2(Y(V))$, we want to show that the image of (\ref{eq:RobevaVeronese}), when restricted to $Y(V)$, linearly spans $S^4V$. But (\ref{eq:RobevaVeronese}) maps an odeco tensor $\sum_{i=1}^k v_i^{\otimes 3}$ to the odeco tensor $\sum_{i=1}^k (v_i|v_i) v_i^{\otimes 4}$, and these tensors clearly span $S^4V$.
\end{proof}

\subsection{The main question}

By Proposition~\ref{prop:Associative}, $X(V)$ contains weakly odeco
tensors, hence in particular the variety $Y(V)$. On the other hand, the
results of \cite{Boralevi17Orthogonal} imply that the set of real points of $X(V)$ equals
$Y(V_\RR)$. This raises the following question. 

\begin{que} \label{que:Central}
For which dimensions $n=\dim(V)$ is $X(V)$ {\em equal} to $Y(V)$?
\end{que}

Our partial answer to this question is as follows. 

\begin{thm} \label{thm:XY1}
For $n=\dim(V) \leq 3$, $X(V)$ equals $Y(V)$. For $n \geq 12$, we
have $X(V) \subsetneq Y(V)$. 
\end{thm}

In fact, the result for $n \leq 3$ is due to Robeva \cite[Figure
2]{Robeva16Orthogonal}. Our contribution is a counterexample to
\cite[Conjecture 3.2]{Robeva16Orthogonal} for $n=12$. 

\subsection{The existence of a unit}

The answer in Theorem~\ref{thm:XY1} is unsatisfactory because of the large interval
of dimensions $n=\dim(V)$ for which we do not know whether $X(V)$ is
strictly larger than $Y(V)$. However, in a certain open subset of $S^3
V$, we {\em do} know precisely where the two stop being equal.

\begin{lm} \label{lm:Unit}
The condition on $T \in X(V)$ that $\mu_T$ has a
multiplicative unit
element is equivalent to the condition that there exists a
$v \in V$ such that the multiplication map $L_x:v \mapsto
\mu_T(x,v)$ is invertible. This is an open condition on $T$. 
\end{lm}

In the case of ordinary tensors, the analoguous condition is called
$1$-genericity; see, e.g., \cite{Jelisiejew22}.

\begin{proof}
We write $xv$ instead of $\mu_T(x,v)$.
For the implication $\Rightarrow$ take $x$ to be the unit
element. For the implication $\Leftarrow$, assume that $L_x$ is
invertible. Then in particular there exists an $e \in V$ such
that $ex=xe=x$. We then find, for any $y \in V$, that 
\[ ey=exL_x^{-1}(y)=xL_x^{-1}(y)=y \]
so that $e$ is a unit element. 

It follows that $(V,\mu_T)$ is not unital if and only if $\det(L_x)=0$
for all $x \in V$; this is a system of degree $n$ polynomial equations
on $T \in S^3(V)$ defining the complement of the unital
locus. 
\end{proof}

We will refer to the variety in $S^3(V)$ defined by the degree $n$
equations in the proof above as the {\em non-invertibility locus}.
If $T \in S^3(V)$ is not in the non-invertibility locus, then $\mu_T$
needs not have a unit element; but it does if furthermore $T$ lies in
$X(V)$---we have used associativity of $\mu_T$ in the proof above.

We define $X^0(V)$ as
\[ X^0(V):=\{T \in X(V) \mid \mu_T \text{ is unital}\}, \]
and similarly for $Y^0(V)$. Note that
$\overline{Y^0(V)}=Y(V)$. This leads to the following
weakening of Question~\ref{que:Central}. 

\begin{que} \label{que:Central2}
For which dimensions $n=\dim(V)$ is $X^0(V)$ {\em equal} to $Y^0(V)$? In
other words, for which dimensions do Robeva's quadrics characterise
the set of limits of strongly odeco tensors in the complement of the
non-invertibility locus?
\end{que}

To answer this question, we will prove the following theorem.

\begin{thm} \label{thm:XY2}
The number of irreducible components of $X^0(\CC^n)$ equals that of the
Gorenstein locus in the Hilbert scheme of $n$ points in $\AA^n_\CC$.
\end{thm}
\begin{proof}
	See Section \ref{sec:XY2}.
\end{proof}

We can now make use of the following result on the irreducibility of Gorenstein loci of Hilbert schemes.

\begin{thm} [{\cite{Casnati15}}]
	The Gorenstein locus of $n$ points in $\AA^d_\CC$ 
	\begin{itemize}
		\item is irreducible if $n \leq 13$, or if $n=14$ and $d \leq 5$.
		\item has 2 irreducible components if $n=14$ and $d \geq 6$.
	\end{itemize}
\end{thm}

Combining the two previous theorems gives complete answer to Question \ref{que:Central2}.

\begin{cor}
	The locus $X^0(\CC^n)$ is irreducible and equal to $Y^0(\CC^n)$ for $n
	\leq 13$, and is not irreducible and not equal to $Y^0(\CC^n)$
	for $n \geq 14$.
\end{cor}

\section{Decomposing (the algebra of) a tensor in $X(V)$}
\label{sec:Decomposition}

\subsection{Motivation}

Recall that if $B$ is a unital finite-dimensional algebra over $\CC$,
then $B$ is isomorphic to a product $B_1 \times \cdots \times B_k$
of local algebras; here $k=|\Spec(B)|$. In this section we
want to establish a similar decomposition for not
necessarily unital algebras.

\subsection{The unital/nilpotent decomposition}

\begin{prop} \label{prop:Decomposition}
Let $T \in X(V)$ and equip $V$ with the corresponding commutative,
associative multiplication $\mu_T$. Then $V$ has a unique
decomposition as direct sum
\[ V_1 \oplus \cdots \oplus V_k \oplus N \]
where the $V_i$ are nonzero and $N$ is potentially zero, such that $N$
and each $V_i$ is an ideal in $(V,\mu_T)$, $(V_i,\mu_T|_{V_i \times
V_i})$ is a local unital algebra for each $i$, and $N$ is a nilpotent
algebra. Furthermore, this unique direct sum decomposition is orthogonal. 

Accordingly, $T$ decomposes as $T_1 + \cdots T_k + T_N$ with
$T_i \in X(V_i)$ and $T_N \in X(N)$; and we have $T \in Y(V)$
if and only if $T_i \in Y(V_i)$ for all $i$ and $T_N \in Y(N)$.
\end{prop}

\begin{proof}
Note that if $x,y \in V$ belong to different factors in any decomposition
of $V$ as a direct sum of ideals, and if the first factor has unit
element $e$, then
\[ (a|b)=(ae|b)=(e|ab)=(e|0)=0; \]
this proves the orthogonality of the decomposition. 

That a {\em unital} finite-dimensional commutative, associative
algebra $V$ has a unique product decomposition into local algebras is
well-known---it is found by taking a decomposition of $1$ into minimal
idempotents $e_i$ satisfying $e_ie_j=\delta_{ij} e_i$ and taking
$V_i:=e_i V$.

To reduce to the unital case, we proceed as follows. If $V$ is
nilpotent, then we set $k:=0$ and $N:=V$. Otherwise, there exists
an element $x \in V$ that is not nilpotent. Let $L_x:V \to V$
be multiplication with $x$. Then there exists an $m$ such that $\im
L_x^m=\im L_x^{m+1}=\ldots$. Set $y:=x^m$ and $W:=yV$. Then $(L_y)|_{yV}$
is invertible, so the ideal $yV$ is unital by Lemma~\ref{lm:Unit}. On the
other hand, we have $V=yV \oplus \ker(L_y)$: indeed, the dimensions of
$yV=\im(L_y),\ker(L_y)$ add up to $n$, and if $L_y(yv)=0$, then $L_y^2
v=0$, so already $L_y v=0$, so $yv=0$. 

Now $\ker L_y$ has strictly
lower dimension than $V$, and by induction we know that $\ker L_y$ is
the direct sum of a nilpotent ideal $N$ and an ideal $I$ that is unital
as an algebra. Then so is $V$: it equals the direct sum of the ideals
$N$ and $yV \oplus I$, where the latter is unital as an algebra. The
decomposition of $V$ into a unital ideal $V_0$ and a nilpotent ideal
$V_1$ is unique, because $V_0$ is the space of elements $v \in V$ for
which there is an idempotent $e \in V$ with $v \in L_e$.

The statements about $T$ are straightforward from the fact that $\mu_T$
is the sum of its restrictions to the ideals $V_i$ and $N$.  Here we
note that the restrictions of $(.|.)$ to the $V_i$ and to $N$ are
non-degenerate, so that the notation $X(V_i),X(N),Y(V_i),Y(N)$ make sense.

\section{$Y(V)$ is a component of $X(V)$}
\label{sec:Component}

\subsection{Motivation}

While, as we will see, $Y(V)$ is in general not {\em equal} to $X(V)$,
at least it is an irreducible component. This was already observed in
\cite{Robeva16Orthogonal}; we paraphrase the argument here.

\subsection{A tangent space computation}

Here, and later in the paper, we will write $e_1,\ldots,e_n$ for the
standard basis of $\CC^n$.

\begin{prop}
For each $V$ equipped with a nondegenerate symmetric
bilinear form, the variety $Y(V)$ of limits of strongly odeco tensors 
is an irreducible component of the variety $X(V)$ defined by
Robeva's quadrics. 
\end{prop}

\begin{proof}
	It suffices to prove that for a suitable tensor $T_0 \in Y(V)$,
	the tangent space to $X(V)$ at $T_0$ has dimension equal to
	$\dim(Y(V))=\binom{n+1}{2}$. We take $V=\CC^n$ and $T_0=E=\sum_{i=1}^n
	e_i^{\otimes 3}$. 
	
	Let us first write Robeva's equations in coordinates: writing $T=\sum_{i,j,k}T_{ijk}e_i \otimes e_j \otimes e_k$, equation (\ref{eq:Robevaxyzw}) becomes
	\begin{equation} \label{eq:RobevaCoords}
	\sum_{r}T_{jkr}T_{i\ell r} = \sum_{r}T_{ijr}T_{k\ell r}
	\end{equation}
	
	The equations defining the tangent space at $E$ are given by substituting $T=E+\varepsilon X$ in (\ref{eq:RobevaCoords}):
	\begin{equation} \label{eq:tangentSpace}
	\delta_{jk}X_{ij\ell} + \delta_{i\ell}X_{ijk} = \delta_{ij}X_{ik\ell} + \delta_{k\ell}X_{ijk}.
	\end{equation}
	
	By taking $i=\ell \neq j \neq k \neq i$ in (\ref{eq:tangentSpace}) we find that $X_{ijk}=0$ for $i,j,k$ pairwise distinct, and by taking $i=\ell \neq j = k$ we find that $X_{iij}=-X_{ijj}$ for all $i \neq j$. But this implies that our tangent space has dimension at most $n+\binom{n}{2}=\binom{n+1}{2}$.	
\end{proof}

\begin{re}
In a similar manner, one finds that all strongly odeco
tensors of tensor rank $n$ are smooth points of $Y(V)$.
\end{re}

\section{Many weakly odeco tensors} \label{sec:Many}

\subsection{Motivation}
In this section, we give our first negative answer to
Question~\ref{que:Central} by showing that, for $n$ sufficiently
large, there are many more weakly odeco tensors than strongly odeco
tensors. 

\subsection{Weakly odeco tensors from isotropic
spaces}

Recall that $V$ is a complex vector space of dimension $n$
equipped with a symmetric bilinear form $(.|.)$.

\begin{prop} \label{prop:ZV}
The variety $X(V)$ contains the union over all (maximal) isotropic subspaces
$U \subseteq V$ of $S^3 U$. This union is an affine variety $Z(V)
\subseteq S^3 V$ of dimension
\[ \binom{\lfloor n/2 \rfloor + 2}{3} + \binom{\lceil n/2 \rceil}{2}. \]
\end{prop}

\begin{proof}
For the first statement, note that if $u_1,\ldots,u_k$ are elements
of an isotropic subspace $U$ of $V$, then $\sum_i u_i^{\otimes 3}$
is weakly odeco, hence in $X(V)$ by Proposition~\ref{prop:Associative}.

There is no harm in restricting our attention to maximal
isotropic subspaces, i.e., those of dimension $\lfloor n/2
\rfloor$. Hence $Z(V)$ is the projection of the incidence variety
\[ \{(U,T) \in \Griso(\lfloor n/2 \rfloor,V) \times S^3 V
\mid T \in S^3 U\} \]
onto the second factor. Since the isotropic Grassmannian is a projective
variety, $Z(V)$ is closed. Furthermore, for $U \subseteq V$ isotropic
of dimension $\lfloor n/2 \rfloor$ and $T \in S^3 U$ {\em concise},
i.e. such that the associated linear map $S^2 U^* \to U$ is surjective,
the fibre over $T$ is the single point $(U,T)$, hence $\dim(Z(V))$ equals
the dimension of the isotropic Grassmannian, which is the second term
above, plus the dimension of $S^3 U$ for a fixed isotropic $U \subseteq V$
of dimension $\lfloor n/2 \rfloor$, which is the first term.
\end{proof}

\begin{re}
Clearly, $\dim(Z(V))$ grows as a cubic (quasi-)polynomial in $n$, whereas
$\dim(Y(V))$ is a quadratic polynomial in $n$. 
Since $X(V) \supseteq Z(V)$, this shows that $X(V) \supsetneq Y(V)$ for all $V$ of sufficiently high dimension. 
In fact, $\dim(Z(V))>\dim(X(V))$ for $n \geq 16$.  However,
we will show with an explicit example that $X(V) \supsetneq Y(V)$ holds
already for $n \geq 12$ (and possibly already for smaller $n$).
\end{re}

\begin{re}
	The variety $Z(V)$ consists precisely of the tensors $T \in S^3V$ whose algebra is $2$-step nilpotent:
	\[
	T \in Z(V) \iff \mu_T(x,\mu_T(y,z)) = 0 \quad \forall x,y,z \in V.
	\]
	One implication is clear: if $T \in Z(V)$, we can write $T=\sum_i{u_i^{\otimes 3}}$ with the $u_i$ isotropic and pairwise orthogonal, and the computation from the proof of Proposition \ref{prop:Associative} gives that $\mu_T(x,\mu_T(y,z)) = 0$. For the other direction we can work in coordinates: write $V=\CC^n$, then the condition 
	\[
	\mu_T(e_i,\mu_T(e_j,e_k))=0 \quad \forall i,j,k \in \{1,\ldots,n\}
	\]
	is equivalent to
	\[
	\sum_{r=1}^n{T_{ijr}T_{k\ell r}} = 0 \quad \forall i,j,k, \ell \in \{1,\ldots,n\}.
	\]
	But this means that that the space $U$ spanned by the columns of $T$ is isotropic.
\end{re}

\section{Unitalisation and de-unitalisation} \label{sec:Unitalisation}

\subsection{Motivation}

It is well known that if an associative algebra $A$, say over $\CC$,
has no multiplicative unit element, then one can turn $A$ into a unital
associative algebra by setting $A':=\CC 1 \oplus A$ and extending the
multiplication on $A$ to $A'$ via $1 a':=a'$ for all $a' \in A'$. In
this section, we describe a process that also extends an invariant
bilinear form.

\subsection{Unitalising algebras with invariant forms} 

Let $A$ be an associative algebra over $C$ equipped with a bilinear form
$(.|.)$ such that $(ab|c)=(a|bc)$ for all $a,b,c \in A$. We
do not require $A$ to be commutative or $(.|.)$ to be symmetric. 

We construct a new algebra
\[ \widetilde{A}:=\CC 1 \oplus A \oplus \CC y \]
with multiplication determined by 
\begin{align*}
&1*x:=x \text{ for all } x \in \widetilde{A},\\
&a*a':=aa'+(a|a')y \text{ for all } a,a' \in A,\\
&a*y:=0, y*a=0  \text{ for all } a \in A, \text{ and}\\
&y*y:=0. 
\end{align*}
We also extend the form $(.|.)$ to $\widetilde{A}$ by requiring that 
\[ (1|1)=(a|1)=(1|a)=(a|y)=(y|a)=(y|y)=0 \text{ for all } a \in A \]
and $(1|y)=(y|1)=1$.

\begin{re}
Let us consider the special case where the multiplication on $A$ is identically zero. If we let $a_1,\ldots,a_n$ be an orthonormal basis of $A$, then the tensor associated to $\tilde{A}$ is equal to
\[
\sum_{i=1}^n{(1 \otimes a_i \otimes a_i +a_i \otimes 1 \otimes a_i + a_i \otimes a_i \otimes 1)} + 1 \otimes 1 \otimes y + 1 \otimes y \otimes 1 + y \otimes 1 \otimes 1. 
\]

This tensor is known as the Coppersmith-Winograd tensor \cite{CoppersmithWinograd}; it has played in central role in the literature on the complexity of matrix multiplication. We refer the reader to \cite{LandsbergComplexity} (in particular Chapter 3.4.9) for an overview. In the notation of the latter reference, the above tensor is denoted $T_{n,CW}$.
\end{re}

\begin{prop}
The algebra $\widetilde{A}$ is associative, and the form $(.|.)$ on
$\widetilde{A}$ is invariant. Furthermore, if $(.|.)$ is nondegenerate or
symmetric on $A$, then its extension to $\widetilde{A}$ has the same property;
and if $A$ is commutative and $(.|.)$ is symmetric, then $\widetilde{A}$
is commutative.
\end{prop}

\begin{proof}
It suffices to prove the identity $a*(b*c)=(a*b)*c$ for $a,b,c$ ranging
over a spanning set of $\widetilde{A}$. If at least one of $a,b,c$ is $1$, then the
identity is immediate. If none of them is $1$ and at least one of them is $y$,
then both sides are zero. So the interesting case is the
case where $a,b,c$ are all in $A$. Then we have
\[ a*(b*c)=a*(bc+(b|c)y)=a(bc)+(a|bc)y+0=a(bc)+(a|bc)y \]
and
\[ (a*b)*c=(ab+(a|b)y)*c=(ab)c+(ab|c)y+0=(ab)c+(ab|c)y. \]
These two expressions are equal by associativity of $A$ and
invariance of $(.|.)$ on $A$. 

Now we turn to the identity $(a*b|c)=(a|b*c)$ for $a,b,c$
ranging over the same spanning set. If $b=1$, then the identity
is immediate. If $b \in A \oplus \CC y$ and $a=1$, then the identity reads
\[ (b|c)=(1|b*c). \]
Now the right-hand side is the coefficient of $y$ in $b*c$.  Write
$b=b'+\beta y$ and $c=\gamma 1 + c' + \delta y$ with $\beta,\gamma,\delta
\in \CC$ and $b',c' \in A$. Then the coefficient of $y$ in $b*c$ equals
$\beta \gamma + (b'|c')$, and this also equals $(b|c)$.
Since $a,c$ play symmetric roles, the identity also holds
when $c=1$. So we are left with the case where $a,b,c \in A
\oplus \CC y$. But then, since $y$ is perpendicular to $A$,
we have 
\[ (a*b|c)=(ab|c)=(a|bc)=(a|b*c), \]
as desired. 

That the extension of $(.|.)$ inherits the properties of symmetry
and non-degeneracy is immediate, and so is the statement about the
commutativity of $\widetilde{A}$.
\end{proof}

\begin{re}
The space $M:=A \oplus \CC y$ is a maximal ideal in $\widetilde{A}$,
and in particular a non-unital subalgebra of $\widetilde{A}$. This
subalgebra has an ideal $\CC y$, and the natural map $A \to M/\CC y$
is an isomorphism of algebras.  We will use this construction below to
{\em de-unitalise} a Gorenstein local algebra in a canonical manner.
\end{re}

\begin{lm}
Suppose that $A$ is commutative and that $(.|.)$ is symmetric. Then $A$
is nilpotent if and only if $\widetilde{A}$ is (unital and) local.
\end{lm}

\begin{proof}
If $A$ is nilpotent, then $M$ consists of elements that are nilpotent in
$\widetilde{A}$, and hence any element not in $M$ is invertible. Conversely,
if $\widetilde{A}$ is local, then $M$ is the unique maximal ideal and its
elements are nilpotent. This implies that $A \cong M/\CC y$
is nilpotent. 
\end{proof}

We now show that each local, unital algebra with an
invariant bilinear form arises as $\widetilde{A}$ for some $A$
equipped with a bilinear form.

\begin{prop} \label{prop:DeUnitalisation}
Let $B$ be a commutative, local, unital algebra of dimension at least $2$
equipped with a nondegenerate invariant symmetric bilinear form. Then
$B \cong \widetilde{A}$ for some nilpotent algebra $A$ equipped with a
nondegenerate invariant symmetric bilinear form.
\end{prop}

\begin{proof}
Let $M$ be the maximal ideal of $B$, and let $d$ be maximal such that
$M^d$ is nonzero. Then $d \geq 1$ since $\dim(B) \geq 2$.

We claim that $M^d$ is one-dimensional. Indeed, if it were at least
two-dimensional, then $1^\perp \cap M$ would contain a nonzero element
$x$. This element would satisfy $(x|1)=0$ and $(x|z)=(xz|1)=(0|1)=0$
for all $z \in M$, contradicting the non-degeneracy of $(.|.)$.

Choose a spanning vector $z \in M^d$. Then $(1|z) \neq 1$, and hence we
may replace $z$ by a (unique) scalar multiple with $(1|z)=1$.
Furthermore, $z^\perp=M$. We define $A$ as the algebra $M/\CC z$ equipped with
the induced symmetric bilinear form. We claim that $\widetilde{A} \cong B$
as algebras with bilinear forms. The isomorphism $\phi:\widetilde{A} \to B$
sends $1 \in \widetilde{A}$ to $1 \in B$, $y \in \widetilde{A}$ to $z \in B$ and $m \in A$ to the unique element $m' \in m+\CC z \subseteq B$ that
satisfies $(1|m')=0$ in $B$. All checks are then straightforward.
\end{proof}

Now let $V$ be an $n$-dimensional complex vector space equipped with a
nondegenerate symmetric bilinear form $(.|.)$. Define $\widetilde{V}:=\CC
1 \oplus V \oplus \CC y$, equipped with the symmetric bilinear form
as above.

Let $T \in X(V)$ and let $V=V_1 \oplus \cdots \oplus V_k \oplus N$
be the decomposition of Proposition~\ref{prop:Decomposition}. Let
$e_i$ be the unit element in $V_i$. Now $\widetilde{V}=\CC 1
\oplus V \oplus \CC y$ is also a
commutative, associative algebra with invariant symmetric bilinear
form $(.|.)$, hence it corresponds to an element $\widetilde{T} \in
X(\widetilde{V})$, which in turn gives a decomposition of $\widetilde{V}$ as in
Proposition~\ref{prop:Decomposition}. The following proposition expresses
the latter decomposition into the former.

\begin{prop} \label{prop:DecompTilde}
We have an orthogonal decomposition 
\[ \widetilde{V}=V_1' \oplus \cdots \oplus V_k' \oplus N' \oplus 0\]
into ideals, where $V_i' \subseteq \widetilde{V}$ is isomorphic
to $V_i$ via the isomorphism 
\[ \phi_i:V_i  \to V_i',\ \phi_i(v):=v+(v|e_i)y \]
and where $N'$ is a local unital algebra spanned by $N$,
$y$, and the unit element
$e_{k+1}:=1-e_1-\cdots-e_k$.
\end{prop}

\begin{proof}
First, $\phi_i$ is clearly injective. It is also an algebra
homomorphism because
\begin{align*} 
&\phi_i(vw)=vw+(vw|e_i)y=vw+(v|we_i)y \\
&=vw+(v|w)y=(v+(v|e_i)y)*(w+(w|e_i)y)=\phi_i(v)*\phi_i(w).
\end{align*}
Now note that if $v,w$ belong to $V_i \neq V_j$,
respectively, then 
\[ \phi_i(v)*\phi_j(w)=(v+(v|e_i)y)*(w+(w|e_j)y)
=vw + (v|w)y = 0. \]
This shows that $V_i'*V_j'=\{0\}$. Similarly, we have $N'*V_i=\{0\}$
for all $i$---e.g., for $v \in V_i$ we have
\[ e_{k+1}*\phi_i(v)=(1-e_1-\cdots-e_k)*(v+(v|e_i)y)
= v+(v|e_i)y-(e_iv+(e_i|v)y)=0. \]
Finally, $e_{k+1}$ is clearly a unit element in $N'$. Indeed, we even
have an isomorphism $\widetilde{N} \to N'$ of unital algebras with invariant
bilinear forms that sends $1$ to $1-e_1-\cdots-e_k$ and $y$ to $y$.
\end{proof}

\begin{prop} \label{prop:UnitalisationMorphism}
The map $T \mapsto \widetilde{T}$ is a morphism from $X(V)$ into
$X(\widetilde{V})$ that maps $Y(V)$ into $Y(\widetilde{V})$.
\end{prop}

We call this morphism the {\em unitalisation morphism}.

\begin{proof}
The first statement is immediate: the algebra structure on $\widetilde{V}$
depends in a polynomial manner on the algebra structure on $V$. For the
last statement, we note that if $T$ is strongly odeco of tensor rank $n$,
then $(V,\mu_T)$ is an orthogonal direct sum of $n$ one-dimensional unital
ideals. By Proposition~\ref{prop:DecompTilde},
$(\widetilde{V},\mu_{\widetilde{T}})$ is then an orthogonal direct sum of $n$
one-dimensional unital ideals and one two-dimensional ideal
which, as an algebra with symmetric bilinear form, is
isomorphic to $\CC[y]/(y^2)$ with the blinear form
determined by $(1|y)=1$. The latter corresponds to the
tensor 
\[ S:=y \otimes y \otimes 1 + y \otimes 1 \otimes y + 1
\otimes y \otimes y \]
from Example~\ref{ex:DimTwo}, hence it is a limit of strongly
odeco tensors. Consequently, by Proposition~\ref{prop:Decomposition},
$\widetilde{T}$ is in $Y(\widetilde{V})$. Since the map $T \mapsto \widetilde{T}$
maps the dense subset of $Y(V)$ of strongly odeco tensors of rank $n$
into $Y(\widetilde{V})$, it maps $Y(V)$ into $Y(\widetilde{V})$.
\end{proof}

\begin{re}
Unfortunately, we see no reason why, if $T \in X(V)$ satisfies $\widetilde{T}
\in Y(\widetilde{V})$, $T$ should be in $Y(V)$. Indeed, the assumption says
that $\widetilde{T}$ is a limit of sums with $n+2$ pairwise orthogonal terms,
and we do not see a natural construction that shows that $T$ is a limit of
sums with $n$ pairwise orthogonal terms; we do not have a
counterexample, though.
\end{re}

\section{Nilpotent counterexamples} \label{sec:Nilpotent}

\subsection{Motivation}
When one studies the Hilbert scheme of $n$ points in a fixed
space for increasing $n$, and $n$ is taken minimal such that
the scheme has more than one irreducible component, then all
components other than the main component parameterise
subschemes supported in a single point. We will establish a
similar result here. 

\subsection{First counterexamples are nilpotent}

\begin{thm}
Let $n=\dim(V)$ be minimal such that $X(V) \neq Y(V)$. Then for all $T
\in X(V) \setminus Y(V)$ the algebra $(V,\mu_T)$ is nilpotent.
\end{thm}

\begin{proof}
Let $T \in X(V) \setminus Y(V)$. Decompose $V=V_1 \oplus \cdots \oplus V_k
\oplus N$ as in Proposition~\ref{prop:Decomposition}, and decompose $T=T_1
+ \cdots +T_k + T_N$ accordingly. By Proposition~\ref{prop:Decomposition},
either some $T_i$ does not lie in $Y(V_i)$, or $T_N$ does not lie
in $Y(N)$. By minimality of $n$, we find that either $k=0$ and
we are done, or else $k=1$ and $N=\{0\}$. In the latter case, by
Proposition~\ref{prop:DeUnitalisation}, the algebra $(V,\mu_T)$ equals
$\widetilde{A}$ for some nilpotent algebra $A$ of dimension $\dim(V)-2$
equipped with a nondegenerate symmetric bilinear form. This means that
$T=\widetilde{S}$ for some tensor $S \in X(A)$.  By minimality of $n$, $S$
lies in $Y(A)$. But then, by Proposition~\ref{prop:UnitalisationMorphism},
$T=\widetilde{S}$ lies in $Y(V)$, a contradiction. Hence $(V,\mu_T)$ is
nilpotent, as claimed.
\end{proof}

\section{Proof of Theorem~\ref{thm:XY1}} \label{sec:XY1}

Let $V$ be a finite-dimensional complex vector space of dimension $n$
and let $(.|.)$ be a nondegenerate symmetric bilinear form on $V$. We
first show that $X(V)$ is not equal to $Y(V)$ when $n=12$.

In \cite{VSPsOfCubic:Jelisiejew} an explicit $14$-dimensional local Gorenstein algebra is
constructed which is not smoothable. Call this algebra $B$, and let
$(.|.)$ be a nondegenerate invariant symmetric bilinear form on $B$. Let
$M$ be the maximal ideal of $B$, and let $M^d$ be its minimal ideal.
Then $A:=M/M^d$ is a nilpotent algebra and since $M^d$ is the radical
of the restriction of $(.|.)$ to $M$, $(.|.)$ induces a nondegenerate
bilinear form on $A$. Note that $\dim(A)=12$, so we may assume that $V$
(with its bilinear form) is the underlying vector space of $A$ (with
its bilinear form). Let $T \in X(V)$ be the tensor corresponding to
the algebra $A$. We claim that $T$ does not lie in $Y(V)$.  Indeed,
if it does, then $T=\lim_{i \to \infty} T_i$ for a convergent sequence
of strongly odeco tensors $T_i$. Applying the unitalisation morphism,
we obtain $\widetilde{T}=\lim_{i \to \infty} \widetilde{T_i}$. Now
$\widetilde{T}$ is the structure tensor of the algebra $\tilde{A}$, which
by (the proof of) Proposition~\ref{prop:DeUnitalisation} is isomorphic
to $B$.

However, by (the proof of) Proposition~\ref{prop:UnitalisationMorphism},
each $\widetilde{T}_i$ is the direct product of $12$ one-dimensional
ideals and one copy of $\CC[x]/(x^2)$. In particular, each
$\widetilde{T}_i$ corresponds to a smoothable algebra, and
$B$ is smoothable, as well. This
contradicts the choice of $B$.
\end{proof}

\section{Proof of Theorem~\ref{thm:XY2}} \label{sec:XY2}

We set $V:=\CC^n$, equipped with the standard symmetric bilinear form
$\beta_0(u,v):=\sum_i u_i v_i$. Recall that $X^0(V)$ is the variety of
tensors corresponding to {\em unital} 
associative algebras on $V$ for which $\beta_0$ is invariant. We want
to show that $X^0(V)$ has the same number of irreducible components as
$\HG$.

\subsection{Locating the unit element}

\begin{lm} \label{lm:Unit2}
The map $u:X^0(V) \to V$ that assigns to a tensor $T$ the unit element
of $(V,\mu_T)$ is a morphism of quasi-affine varieties.
\end{lm}

\begin{proof}
For given $T$, the unit element $u=u(T)$ is the solution to the system
of linear equations $\mu_T(u,e_i)=e_i$ for $i=1,\ldots,n$. For each $T
\in X^0(V)$, this system has a unique solution. This means that we can
cover $X^0(V)$ with open affine subsets in which some subdeterminant
of the coefficient matrix has nonzero determinant, and on such an open
subset the map $u$ is morphism with a formula in which that determinant
appears in the denominator. These morphisms glue to a global morphism $u$.
\end{proof}

\subsection{A map from $X^0(V)$ to the Hilbert scheme}

We write $R:=\CC[x_1,\ldots,x_n]$, denote by $\cH$ the Hilbert scheme
of $n$ points in $\AA^n$, and by $\HG$ the open subscheme of $\cH$
parameterising Gorenstein schemes. In fact, since we care only about
irreducible components, we may and will replace both of these by the
corresponding reduced subvarieties, and we will only speak of $\CC$-valued
points of these varieties. Points in $\cH$ will be regarded as ideals in
$R$ of codimension $n$. To define a morphism from an affine variety $B$
over $\CC$ to $\cH$, it suffices to indicate a subscheme of $B \times
\AA^n$ (product over $\CC$), flat over $B$, such that the fibre over
each $b \in B$ is defined by such a codimension-$n$ ideal.

Take $B=X^0(V)$. A tensor $T \in X^0(V)$ gives rise to the ideal
$I_T:=\ker(\phi_T)$, where $\phi_T:K[x_1,\ldots,x_n] \to (V,\mu_T)$
is the kernel of the homomorphism of associative algebras that maps
$x_i$ to $e_i$ and $1$ to the unit element $u(T)$ from
Lemma~\ref{lm:Unit2}. The ideals $I_T$ have vector space
codimension $n$ in $T$ and together define a subscheme of $X^0(V) \times
\AA^n$ flat over $X^0(V)$. Hence we have described a morphism $\Phi:X^0(V)
\to \cH$. Since any algebra corresponding to a tensor in $X^0(V)$
has a nondegenerate invariant bilinear form, $\Phi(T) \cong (V,\mu_T)$
is Gorenstein for each $T \in X^0(V)$, so $\Phi$ is a morphism $X^0(V)
\to \HG$.  We want to use $\Phi$ to compare irreducible components of
$\HG$ and $X^0(V)$.  However, the map $\Phi$ is not an isomorphism,
so some care is needed for this.  We first describe the image of
$\Phi$; the following is immediate.

\begin{lm} \label{lm:ImPhi}
The image of $\Phi$ consists of all codimension-$n$ ideals $I \in \HG$
such that $x_1,\ldots,x_n \in R$ map to a basis of $R/I$ and moreover
the bilinear form on $R/I$ for which this basis is orthonormal is
invariant for the multiplication in $R/I$. \hfill \qed
\end{lm}

The following lemma shows that, as far as irreducible components are
concerned, it is no real restriction to consider ideals $I$ modulo which
$x_1,\ldots,x_n$ is a basis. 

\begin{lm} \label{lm:HG0}
The locus $\HG_0$ in $\HG$ of ideals $I$ in $R$ for which $x_1,\ldots,x_n$
maps to a vector space basis of $R/I$ is open and dense in $\HG_0$.
Consequently, $\HG_0$ has the same number of irreducible components as
$\HG$. 
\end{lm}

This is well-known to the experts---see \cite[Remark
4.5]{PoonenModuli}---but we include a quick proof.

\begin{proof}
The condition on $I$ can be expressed by the non-vanishing of certain
determinants; this shows that $\HG_0$ is open. For density, suppose
that some component $C$ of $\HG$ does not meet $\HG_0$, and let $I_0$
be a point in $C$ such that the image of $\langle x_1,\ldots,x_n
\rangle_\CC$ in $R/I_0$ has maximal dimension, say $m<n$. 
After a linear change of coordinates (which preserves
all components of $\HG$ and hence in particular $C$), we may assume that $x_{m+1},\ldots,x_n$ are
in $I_0$, and there exists a monomial $r$ in $x_1,\ldots,x_m$ of degree
$\neq 1$ such that $x_1,\ldots,x_m,r$ are linearly independent in $R/I_0$.

Now, for $a \in \CC$, consider the nonlinear automorphism $\psi_a:R \to R$
that maps all $x_i, i \neq n$ to themselves but $x_n$ to $x_n+a\cdot r$. The
map $(a,I) \mapsto \psi_a^{-1}(I)$ defines an action of the additive
group $(\CC,+)$ on $\HG$, and since the additive group is irreducible,
this action preserves all components of $\HG$. Since, for $a \neq 0$,
$x_1,\ldots,x_m,x_n+a r$ are linearly independent modulo $I_0$, their
pre-images $x_1,\ldots,x_m,x_n$ under $\psi_a$ are linearly independent
modulo $I_a:=\psi_a^{-1}(I_0)$. Since $I_a$ is in $C$, this contradicts
the maximality assumption in the choice of $I_0$.

The last statement is now immediate.
\end{proof}

\subsection{Varieties $Z_2 \to Z_1 \to \HG_0$ with the same number of
components}

In what follows, we will identify $V$ with the space in $R$ spanned
by the variables $x_1,\ldots,x_n$, via the identification $e_i \mapsto
x_i$. Each point in $\HG_0$ defines a unital, commutative, associative
algebra structure on $V$.  The structure constant tensor in $(S^2 V^*)
\otimes V$ of this algebra does not necessarily lie in $X^0(V)$, though,
because the standard form $\beta_0$ may not be invariant for it.

\begin{lm} \label{lm:Z1}
Let $Z_1$ be the subvariety 
\[ \{(I,[\beta]) \in \HG_0 \times \PP(S^2 V^*) \mid \beta \text{ is
invariant for } R/I\} \subseteq \HG_0 \times \PP(S^2 V^*). \]
Then the projection $Z_1 \to \HG_0$ is surjective and induces a bijection
on irreducible components.
\end{lm}

\begin{proof}
Indeed, every (possibly degenerate) invariant bilinear form on $R/I$ is of the form
$\beta(r,s)=\ell(rs)$ for a unique linear form $\ell \in (R/I)^*$,
namely, the form $\ell(r):=\beta(1,v)$.
Moreover, since for $I \in \HG_0$ the space $V$ is a vector space
complement of $I$ in $R$, the natural map $(R/I)^* \to V^*$ is a linear
bijection. We conclude that, in fact, $Z_1$ is isomorphic to $\HG_0
\times \PP(V^*)$ via the map that sends $(I,[\beta])$ to $(I,[v \mapsto
\beta(1,v)])$. So each component of $Z_1$ is just a component of $\HG_0$
times the projective space $\PP(V^*)$.
\end{proof}

\begin{lm} \label{lm:Z2}
Let $Z_2$ be the subvariety
\[ \{((I,[\beta]),g) \in Z_1 \times \GL(V) \mid g [\beta]=[\beta_0] \}
\subseteq Z_1 \times \GL(V). \]
Then the projection $Z_2 \to Z_1$ has dense image and induces a bijection
between irreducible components.
\end{lm}

\begin{proof}
For the first statement, if $I \in \HG_0$, then by definition there are
\emph{nondegenerate} invariant bilinear forms on $R/I$. These correspond to
a dense open subset of $\PP(V^*)$ via the correspondence in the proof
above. This shows that $Z_2 \to Z_1$ has dense image. This image, $U$,
is open in $Z_1$. 

Next we claim that, in the analytic topology, $Z_2 \to U$ is a
fibre bundle with fibre the group $\CC^* \cdot \OO(\beta_0) \subseteq
\GL(V)$; here $\OO(\beta_0)$ is the orthogonal group of the form
$\beta_0$. To see
this, it consider a point $(I,[\beta_1]) \in U$. By definition of
$U$, there exists a $g_1 \in \GL(V)$ such that $g_1 [\beta_1] =
[\beta_0]$. Furthermore, there exists a holomorphic map $\gamma$
defined in an open neighbourhood $\Omega$ in $\PP S^2 V^*$ of
$[\beta_0]$ to $\GL(V)$ such that $\gamma([\beta_0])=\id_V$ and
$\gamma([\beta])[\beta]=[\beta_0]$ for all $[\beta] \in \Omega$. Essentially, $\gamma([\beta])$ is
found by the Gram-Schmidt algorithm---note that in this algorithm one
has to divide by square roots of complex numbers, which, since
$[\beta]$ is close to $[\beta_0]$, are close to $1$; this
can be done holomorphically.

Now the map 
\begin{align*} 
&(U \cap (\HG_0 \times g_1^{-1}\Omega)) \times (\CC^* \cdot \OO(\beta_0)) \to 
Z_2, \\
&((I,[\beta]),g) \mapsto ((I,[\beta]),g \cdot \gamma(g_1 [\beta]) \cdot g_1)
\end{align*}
trivialises the map $Z_2 \to Z_1$ over an open neighbourhood of
$(I,[\beta_1])$; here we use that $\CC^* \cdot \OO(\beta_0)$ is the
stabiliser of $[\beta_0]$ in $\GL(V)$.

Now since $Z_2 \to U$ is a fibre bundle with {\em irreducible} fibre $\CC^*
\cdot \OO(V)$---this is where it is important that we work with the
projective space $\PP S^2 V^*$ rather than $S^2 V^*$; the orthogonal
group $\OO(V)$ itself has two components!---that map induces a bijection
between irreducible components. 
\end{proof}

\subsection{Completing the proof}

\begin{proof}[Proof of Theorem~\ref{thm:XY2}.]
Recall that $X^0(V)$ parameterises the unital associative, commutative
algebra structures on $V$ such that $\beta_0$ is invariant for the
multiplication. Now consider the map
\[ \GL(V) \times X^0(V) \to \HG_0, (g,T) \mapsto g \cdot \Phi(T)
\]
By Lemma~\ref{lm:ImPhi}, this map is surjective. Since $\GL(V)$ is
irreducible, the left-hand side has as many irreducible components as
$X^0(V)$. This shows that $X^0(V)$ has at least as many irreducible
components as $\HG_0$, hence as $\HG$ by Lemma~\ref{lm:HG0}. 

For the converse, by Lemmas~\ref{lm:Z1}, \ref{lm:Z2} $\HG_0$ has as
many irreducible components as $Z_2$. Now we claim that the morphism 
\begin{align*} 
	Z_2 &\to (S^2 V^*) \otimes V \\ ((I,[\beta]),g) &\mapsto 
\text{ the structure constant tensor of } R/(g \cdot I)
\end{align*}
has as image the variety $X^0(V)$. Indeed, if $((I,[\beta]),g)$ lies
in $Z_2$, then $\beta$ is an invariant symmetric bilinear form for the
multiplication on $R/I$, and $g \beta \in \CC^* \cdot \beta_0$ is an
invariant symmetric bilinear form for the multiplication on $R/(g \cdot
I)$; this therefore corresponds to an element in $X^0(V)$. We conclude
that the number of components of $X^0(V)$ is also at most that of
$\HG$. This concludes the proof.
\end{proof}

\subsection{Cubic dimension growth for the Gorenstein locus}
\label{sec:Cubic}

We conclude this paper with an observation on the dimension of 
$\dim(\HG)$.

\begin{prop}
The dimension of $\HG$, and hence that of the Hilbert scheme $\cH$ of $n$
points in $\AA^n$, is lower-bounded by a cubic polynomial in $n$ for $n
\to \infty$.
\end{prop}

Note that this was already known for $\cH$ by
\cite[Theorem 9.2]{PoonenModuli}. The algebras constructed there are of the
form $A:=\CC[x_1,\ldots,x_d]/(V+\fm^3)$ where $\fm$ is the maximal ideal
$(x_1,\ldots,x_d)$ and where $V \subseteq \fm^2/\fm^3$ has the correct
codimension $r=n-1-d$ for this quotient to have dimension $n$. Since any
$1$-dimensional subspace of $\fm^2/(V+\fm^3)$ is a minimal ideal in $A$,
$A$ is not Gorenstein unless $\fm^2/(V+\fm^3)$ is one-dimensional, in
which case $r=1$. However, to obtain cubic behaviour in $n$, one needs $r$
to grow linearly with $d$. So the cubic-dimensional locus in
\cite{PoonenModuli}
is a non-Gorenstein part of the Hilbert scheme.

\begin{proof}
The unitalisation morphism sends $X(\CC^{n-2})$ into $X^0(\CC^{n})$
by Proposition~\ref{prop:UnitalisationMorphism}, and it does so
injectively. This means that the latter variety has dimension at least
that of $Z(\CC^{n-2})$, which is lower-bounded by a cubic polynomial by
Proposition~\ref{prop:ZV}. Furthermore, the morphism $\Phi:X^0(\CC^{n})
\to \HG$ is also injective. 
\end{proof}

\begin{re}
The coefficient of $n^3$ in $\dim(Z(\CC^n))$ equals $\frac{1}{48}$,
which is considerably smaller than the coefficient $\frac{2}{27}$ in
\cite{PoonenModuli} for the lower bound on the dimension of the Hilbert
scheme of $n$ points in $\AA^n$. We do not know whether the
$\frac{1}{48}$ can be improved.
\end{re}

\bibliographystyle{alpha}
\bibliography{odecoQuadrics}

\begin{thebibliography}{BDHR17}

\bibitem[BDHR17]{Boralevi17Orthogonal}
Ada Boralevi, Jan Draisma, Emil Horobe\c{t}, and Elina Robeva.
\newblock Orthogonal and unitary tensor decomposition from an algebraic
  perspective.
\newblock {\em Israel J. Math.}, 222(1):223--260, 2017.

\bibitem[CJN15]{Casnati15}
Gianfranco Casnati, Joachim Jelisiejew, and Roberto Notari.
\newblock Irreducibility of the {Gorenstein} loci of {Hilbert} schemes via ray
  families.
\newblock {\em Algebra Number Theory}, 9(7):1525--1570, 2015.

\bibitem[CW90]{CoppersmithWinograd}
Don Coppersmith and Shmuel Winograd.
\newblock Matrix multiplication via arithmetic progressions.
\newblock {\em J. Symbolic Comput.}, 9(3):251--280, 1990.

\bibitem[Fla68]{Flanigan68}
Francis~J. Flanigan.
\newblock Algebraic geography: {Varieties} of structure constants.
\newblock {\em Pac. J. Math.}, 27:71--79, 1968.

\bibitem[Jel18]{VSPsOfCubic:Jelisiejew}
Joachim Jelisiejew.
\newblock {VSP}s of cubic fourfolds and the {G}orenstein locus of the {H}ilbert
  scheme of 14 points on {A6}.
\newblock {\em Linear Algebra and its Applications}, 557:265--286, 2018.

\bibitem[JLP22]{Jelisiejew22}
Joachim Jelisiejew, Joseph~M. Landsberg, and Arpan Pal.
\newblock Concise tensors of minimal border rank.
\newblock 2022.
\newblock Preprint, \verb+2205.05713+.

\bibitem[Koi21]{Koiran21}
Pascal Koiran.
\newblock Orthogonal tensor decomposition and orbit closures from a linear
  algebraic perspective.
\newblock {\em Linear Multilinear Algebra}, 69(13):2353--2388, 2021.

\bibitem[Lan17]{LandsbergComplexity}
Joseph~M. Landsberg.
\newblock {\em Geometry and complexity theory}, volume 169 of {\em Cambridge
  Studies in Advanced Mathematics}.
\newblock Cambridge University Press, Cambridge, 2017.

\bibitem[Poo08]{PoonenModuli}
Bjorn Poonen.
\newblock The moduli space of commutative algebras of finite rank.
\newblock {\em Journal of the European Mathematical Society}, page 817–836,
  2008.

\bibitem[Rob16]{Robeva16Orthogonal}
Elina Robeva.
\newblock Orthogonal decomposition of symmetric tensors.
\newblock {\em SIAM J. Matrix Anal. Appl.}, 37(1):86--102, 2016.

\end{thebibliography}

\end{document}